\documentclass[11pt, a4paper]{amsart}

\usepackage{amsmath,amssymb,amscd,amsfonts}
\usepackage[textwidth=360pt,textheight=615pt]{geometry}

\newtheorem{thm}{Theorem}[section]
\newtheorem{lem}[thm]{Lemma}
\newtheorem{rem}[thm]{Remark}
\newtheorem{prop}[thm]{Proposition}
\newtheorem{cor}[thm]{Corollary}

\newtheorem{claim}[thm]{Claim}

\newtheorem{ques}[thm]{Question}

\newcommand\ct{{\text{\rm ct}}}
\newcommand\lct{{\text{\rm lct}}}

\newcommand\cE{{\mathcal{E}}}

\newcommand\cX{{\mathcal{X}}}
\newcommand\cY{{\mathcal{Y}}}
\newcommand\cZ{{\mathcal{Z}}}

\newcommand\bR{{\mathbb R}}

\newcommand\bC{{\mathbb C}}
\newcommand\bQ{{\mathbb Q}}
\newcommand\bN{{\mathbb N}}

\newcommand\bP{{\mathbb P}}
\begin{document}
\title{Accumulation points on 3-fold canonical thresholds}
\author{Jheng-Jie Chen}
\address{\rm Department of Mathematics, National Central University, Taoyuan City, 320, Taiwan}
\email{jhengjie@math.ncu.edu.tw}

\maketitle

\begin{abstract}  
We obtain that the nonzero accumulation points of the set of 3-fold canonical thresholds $\ct(X,S)$ are precisely $1/k$ where $k\ge 2$ is an integer and $S$ is an effective integral divisor of a projective 3-fold $X$ with only terminal singularities.  
Moreover, we generalize the ascending chain condition for the set of 3-fold canonical thresholds to pair.
\end{abstract}

\section{introduction}

We work over the complex number field $\mathbb{C}$.

For a log canonical pair $(X,B)$ and an effective $\bR$-Cartier divisor $S$, the log canonical threshold of $S$ with respect to $(X,B)$ is defined by
$$\lct(X,B;S):=\sup\{t\in \bR\ |\ (X,B+tS) \textup{ is log canonical}\}.$$
The log canonical threshold is a fundamental invariant in the study of birational geometry (See \cite{Kol97,Kol08,MP04}).
Recently, de Fernex, Ein and Musta\c{t}\u{a} prove Shokurov's ascending chain condition (ACC) conjecture for log canonical thresholds on varieties that are locally complete intersection \cite{dFEM}. Then,  
Hacon, Mckernan and Xu prove the general case in \cite[Theorem 1.1]{HMX}. 

Given an ACC set. It is natural to study the set of accumulation points. In \cite{Kol97,Kol08}, Koll\'{a}r conjectures that the set of the accumulation points in dimension $n$ equals to the set of log canonical thresholds in dimension $n-1$. In \cite{MP04}, McKernan and Prokhorov prove Koll\'{a}r conjecture in dimension $3$. 
In arbitrary dimension, de Fernex and Musta\c{t}\u{a} and Koll\'{a}r independently show that Koll\'{a}r's conjecture holds for smooth case (see \cite{dFM,Kol08}). Then, Hacon, Mckernan and Xu prove the general case in \cite[Theorem 1.11]{HMX}.

In this paper, we consider its analog. Let $X$ be a $\bQ$-factorial projective variety with at worst canonical singularities and $S$ be an integral and effective divisor. The canonical threshold of the pair $(X,S)$ is defined to be $$\ct(X;S):=\sup\{t\in \bR\ |\ \textup{the pair }(X,tS) \textup{ is canonical}\}.$$ For every positive number $n$, we define the set of canonical threshold by
$$\mathcal{T}_{n}^{\textup{can}} := \{\ct(X, S) | \dim X = n \}.$$

It is known and easy to compute $\mathcal{T}_{2}^{\textup{can}}=\{\frac{1}{k}\}_{k\in \mathbb{N}}$ where $\mathbb{N}$ denotes the set of positive integers. However, it is very difficult to describe $\mathcal{T}_{n}^{\textup{can}}$ for $n\ge 3$.
In the case $n=3$, Prokhorov shows that the largest canonical threshold (less than 1) is $\frac{5}{6}$ (resp. $\frac{4}{5}$) when $X$ is smooth (resp. singular) in \cite{Prok08}. 
Then, Stepanov studies the set $$\mathcal{T}_{3, \textup{sm}}^{\textup{can}}:=\{\ct(X,S)\ |\ \dim X=3, X\textup{ is smooth}\}.$$ He obtains that $\mathcal{T}_{3, \textup{sm}}^{\textup{can}}$ satisfies the ACC and establishes the explicit formula for $\ct_P(X,S)$ when $P\in S$ is a Brieskorn singularity in \cite{Stepanov}. Then, in \cite{3ct}, the author proves the ACC for $\mathcal{T}_{3}^{\textup{can}}$ (by applying Stepanov's argument) and obtains that the intersection $\mathcal{T}_{3}^{\textup{can}}\cap (\frac{1}{2},1)$ coincides with $\{ \frac{1}{2}+\frac{1}{k}\}_{k \ge 3} \cup \{ \frac{4}{5}\}$. We note that the set $\mathcal{T}_{3}^{\textup{can}}\cap (\frac{1}{2},1)$ has the accumulation point $1/2$.
 
The aim of this paper is to prove the following:
\begin{thm}\label{accpt}
The set of accumulation points of $\mathcal{T}_{3}^{\textup{can}}$ consists of $\{0\}\cup\{1/k\}_{k\in \mathbb{N}_{>1}}$. 
\end{thm}

Note that Theorem \ref{accpt} is analogous to the main theorem in \cite{MP04} and our argument provides an alternative method for the ACC for $\mathcal{T}_{3}^{\textup{can}}$. This gives partial answer to Question 8.1 in \cite{3ct}.

Once $\mathcal{T}_{3}^{\textup{can}}$ satisfies the ACC, it might not be difficult for experts to obtain the generalization of the ACC for $\mathcal{T}_{3}^{\textup{can}}$ to pair. That is, the ACC for $1$-lc thresholds holds in dimension $3$ (See \cite{BS} for the notion of $1$-lc thresholds). We provide its argument here for reader's convenience.
In order to give the statement, we fix the following similar notions as in \cite{HMX}. For a $\bQ$-Cartier effective divisor $S$ on a $\bQ$-factorial projective 3-fold $X$, we define 3-fold canonical threshold of $S$ with respect to the canonical pair $(X,B)$ by $$\ct(X,B;S):=\sup\{t\in \bR\ |\ (X,B+tS) \textup{ is canonical}\}.$$
For $I\subset [0,1]$ and $J\subset \bR_{>0}$, we define 
$$CT_3(I;J):=\{\ct(X,B;S)\ |\ B\ (\mbox{resp.}\  S\text{) on $X$ has coefficients in }I\ (\mbox{resp.}\  J )\}.$$ 
\begin{thm}[ACC for 3-fold canonical thresholds (pair version)]\label{ACCctpair}
Keep notions above.
Suppose both $I\subset [0,1]$ and $J\subset \bR_{>0}$ satisfy the descending chain condition (DCC). 
Then the set $CT_3(I;J)$ satisfies the ACC.
\end{thm}

Recently, Jihao Liu informed the author that they also obtain Theorem \ref{accpt} and Theorem \ref{ACCctpair} in the work \cite[Theorem 1.7, Theorem 1.8]{HLL} where Theorem \ref{accpt} above plays an important role in proving the ACC for minimal log discrepancy(mld) for terminal threefolds (See \cite[page 5, Theorem 1.1, Theorem 4.1]{HLL}).

To prove Theorem \ref{accpt} and Theorem \ref{ACCctpair} above, we adopt the classification of divisorial contractions that contract divisors to points, due to Hayakawa, Kawakita, Kawamata, and many others (cf. \cite{Ka,Haya99,Haya00,Kawakita01,Kawakita02,Kawakita05}).
The argument in the classification of $\mathcal{T}_{3}^{\textup{can}}\cap (\frac{1}{2},1)$ in \cite[Theorem 1.3]{3ct} enables us to obtain Theorem \ref{accpt}. Then, Theorem \ref{ACCctpair} follows from Stepanov's argument of the ACC for $\mathcal{T}_{3}^{\textup{can}}$ in \cite[Theorem 1.2]{3ct} with more careful considerations.  

Note that the proof of Theorem \ref{accpt} is different from that of \cite[Theorem 4.8]{HLL} as they bypass the $cA/n$ and the $cD/2$ cases by taking cyclic covers. In fact, we observe some inequalities by comparing the weight $w$ that computes the canonical threshold with the weight $w_{a-n}$ where $a$ (resp. $a-n$) is the weighted discrepancy of the weighted blow up of $w$ (resp. $a-n$) and $n$ denotes the index of the center (See the arguments in Proposition \ref{cAnacc} and Case 1 in Proposition \ref{cD2acc}).

\noindent
{\bf Acknowledgement.}
The author was partially supported by
NCTS and MOST of Taiwan.
He expresses his gratitude to Professor Jungkai Alfred Chen for extensive help and invaluable discussions and suggestions.
He would like to thank Dr. Jihao Liu for asking him a question of the generalizing version of the ACC conjecture for canonical thresholds when the paper \cite{3ct} was posted on arXiv December 2019. He is thankful to Dr. Jihao Liu for recent communication on February 9th and 10th, 2022 and warm considerations.

\section{Proof of Theorem \ref{accpt}}
We follow notions in \cite[Sections 1,2]{3ct}. In this section, we denote by $S$ an integral and effective divisor on $X$ where $S$ is defined by a formal power series
$(f=0)$ locally.  

For $i\in \bN$, let $P_i\in X_i$ be the cyclic quotient $\bC^3/1/2(1,1,1)$ and $w=\frac{1}{2}(1,1,1)$ and let $S_i$ be an effective Weil divisor through $o\in X_i$ with defining semi-invariant formal power series $f_i$ such that $2w(f_i)=i$. Then $0$ is the canonical thresholds of $\{\ct_{P_i}(X_i,S_i)\}=\{\frac{1}{i}\}_{i\in \bN}$ by \cite{Ka}.

From the decomposition in \cite{3ct}
$$\mathcal{T}_{3}^{\textup{can}} = \aleph_4 \cup \mathcal{T}_{3,sm}^{\textup{can}} \cup \mathcal{T}_{3, cA}^{\textup{can}} \cup \mathcal{T}_{3, cA/n}^{\textup{can}} \cup \mathcal{T}_{3,cD}^{\textup{can}} \cup \mathcal{T}_{3, cD/2}^{\textup{can}} \eqno{\dagger}
,$$
Theorem \ref{accpt} follows from Propositions \ref{smacc}, \ref{cAacc}, \ref{cAnacc}, \ref{cDacc} and \ref {cD2acc}.

\begin{prop}\label{smacc}
Let $k\ge 2$ be an integer. Then the accumulation point of $\mathcal{T}_{3,sm}^{\textup{can}}\cap (\frac{1}{k},\frac{1}{k-1})$ is $\frac{1}{k}$. Moreover, if $\ct\in \mathcal{T}_{3,sm}^{\textup{can}}\cap (\frac{1}{k},\frac{1}{k-1})$, then $\ct=\frac{1}{k}+\frac{q}{p}$ with $p, q$ positive integers and $q\leq 2k$.
\end{prop}
\begin{proof}
It follows from \cite[Theorem 3.6]{Stepanov} that $\frac{1}{k}$ is an accumulation point of $\mathcal{T}_{3,sm}^{\textup{can}}\cap (\frac{1}{k},\frac{1}{k-1})$.

Suppose that $\ct\in \mathcal{T}_{3,sm}^{\textup{can}}\cap (\frac{1}{k},\frac{1}{k-1})$ is a canonical threshold computed by a weighted blow up with weights $w=(1,\alpha,\beta)$ with  $1 \le \alpha <\beta $. If $\alpha>1$, then by \cite[Proposition 3.3]{3ct}, 
$\frac{1}{k} < \ct \le \frac{1}{\alpha}+\frac{1}{\beta}$. It follows that  $\alpha < 2k$.

Compare the weight $w$ with the weight $w':=(1,\alpha,\beta-1)$. Note that the exceptional set of the weighted blow up of weight $w'$ is isomorphic to the weighted projective space $\bP(1,\alpha,\beta-1)$ which is clearly irreducible. 
By the inequalities in \cite[Lemma 2.1]{3ct}, we have
$$ \lfloor \frac{\alpha+\beta-1}{\alpha+\beta} m \rfloor  \ge \lceil \frac{\beta-1}{\beta} m \rceil=m-\lfloor \frac{m}{\beta}\rfloor.$$
Since $\ct=\frac{\alpha+\beta}{m}\in (\frac{1}{k},\frac{1}{k-1})$, one sees $\frac{m}{\beta}\ge \lfloor \frac{m}{\beta} \rfloor \ge \lceil \frac{m}{\alpha+\beta}\rceil =k$. 
In particular, $\beta k\le m<\alpha k+\beta k$. Let $p,q$ be two positive relatively prime integers with
$$\frac{q}{p}=\ct-\frac{1}{k}=\frac{\alpha+\beta}{m}-\frac{1}{k}=\frac{\alpha k+\beta k-m}{mk}\leq \frac{\alpha }{m}.$$
Then $q\leq \alpha \le 2k$.


Now, suppose $\{\ct_i=\frac{1}{k}+\frac{q_i}{p_i}\}$ is a sequence converging to a real number $x$ where each $\ct_i\in \mathcal{T}_{3,sm}^{\textup{can}}\cap (\frac{1}{k},\frac{1}{k-1})$ with positive integers $p_i, q_i$ and $q_i\leq 2k$. Since each $q_i\leq 2k$, by passing to a subsequence, one may assume $q_i=q$ is fixed for all $i$. In particular, $x=\lim_{i\to \infty} (\frac{1}{k}+\frac{q}{p_i})=\frac{1}{k}$.
This completes the proof.
\end{proof}

The argument for $cA$ case is similar. For the readers' convenience, we provide the proof.

\begin{prop}\label{cAacc}
Let $k\ge 2$ be an integer. If $x$ is an accumulation point of $\mathcal{T}_{3,cA}^{\textup{can}}\cap (\frac{1}{k},\frac{1}{k-1})$, then $x=\frac{1}{k}$. Moreover, if $\ct\in \mathcal{T}_{3,cA}^{\textup{can}}\cap (\frac{1}{k},\frac{1}{k-1})$, then $\ct=\frac{1}{k}+\frac{q}{p}$ with $p, q$ positive integers and $q\leq 2k$.
\end{prop}
\begin{proof}
Suppose that $\ct:=\ct(X,S) \in \mathcal{T}_{3, cA}^{\textup{can}}\cap (\frac{1}{k},\frac{1}{k-1})$ is a canonical threshold realized by a divisorial contraction $\sigma: Y \to X$. Theorem 1.2(i) in \cite{Kawakita05} shows that there exists an analytical identification $P\in X \simeq o\in (\varphi=xy+g(z,u)=0)$ in $\bC^4$ where $o$ denotes the origin of $\bC^4$ and $\sigma$ is a weighted blow up of weight $w=wt(x,y,z,u)=(r_{1},r_{2},a,1)$ such that $w(g(z,u))=r_{1}+r_{2}=ad$ where $r_1,r_2,a,d$ are positive integers. 


Without loss of generality, we assume $r_1 \le r_2 $.
If $r_1>1$, then by \cite[Proposition 4.2]{3ct}, 
$\frac{1}{k} < \ct \le \frac{1}{r_1}+\frac{1}{r_2}$. It follows that  $r_1 < 2k$.

Compare the weights $w$ with the weights $w_{a-1}:=(r_1,r_2-d,a-1,1)$. 
By \cite[Lemmas 2.1, 4.1]{3ct}, we have
$$ \lfloor \frac{a-1}{a} m \rfloor  \ge \lceil \frac{r_2-d}{r_2} m \rceil=m-\lfloor \frac{dm}{r_2}\rfloor.$$
Since $\ct=\frac{a}{m}\in (\frac{1}{k},\frac{1}{k-1})$, one sees $\frac{r_1+r_2}{r_2}\cdot\frac{m}{a}=\frac{dm}{r_2}\ge \lceil \frac{m}{a}\rceil =k$. In particular, $r_2k\le dm<r_1k+r_2k$. Let $p,q$ be two positive relatively prime integers with
$$\frac{q}{p}=\ct-\frac{1}{k}=\frac{a}{m}-\frac{1}{k}=\frac{r_1+r_2}{dm}-\frac{1}{k}=\frac{r_1k+r_2k-dm}{dmk}\leq \frac{r_1}{dm}.$$
Then $q\leq r_1\le 2k$. In particular, the only accumulation point of $\mathcal{T}_{3,cA}^{\textup{can}}\cap (\frac{1}{k},\frac{1}{k-1})$ is $\frac{1}{k}$ as in the same argument in Proposition \ref{smacc}.
This completes the proof.

\end{proof}


The argument for $cA/n$ case is more complicated. 

\begin{prop}\label{cAnacc}
Let $k\ge 2$ be an integer. If $x$ is an accumulation point of $\mathcal{T}_{3,cA/n}^{\textup{can}}\cap (\frac{1}{k},\frac{1}{k-1})$, then $x=\frac{1}{k}$. 
\end{prop}
\begin{proof}
Suppose that the canonical threshold $\ct:=\ct(X,S)\in \mathcal{T}_{3,cA/n}^{\textup{can}}\cap (\frac{1}{k},\frac{1}{k-1})$ is computed by a weighted blow up $\sigma: Y \to X$ over a $cA/n$ point $P \in X$ with weighted discrepancy $a \ge 5$.
Theorem 1.2(i) in \cite{Kawakita05} shows that there exists an analytical identification $P\in X \simeq o\in (\varphi \colon xy+g(z^n,u)=0)$ in $\bC^4/\frac{1}{n}(1,-1,b,0)$ where $o$ denotes the origin of $\bC^4/ \frac{1}{n}(1,-1,b,0)$ and $\sigma$ is a weighted blow up of weight $w=wt(x,y,z,u)=\frac{1}{n} (r_{1},r_{2},a,n)$ satisfying the following:
\begin{itemize}
\item $nw(\varphi)=r_1+r_2=adn$ where $r_1,r_2,a,d,n$ are positive integers.
\item $z^{dn}\in g(z^n,u)$.
\item $a\equiv br_1$ (mod $n$) and $0<b<n$.
\item $\gcd(b,n)=\gcd(\frac{a-br_1}{n},r_1)=\gcd(\frac{a+br_2}{n},r_2)=1$ (See \cite[Lemma 6.6]{Kawakita05}).
\item $S$ is defined by the semi-invariant formal power series $f=0$ locally so that $\ct(X,S)=\frac{a}{m}$, where $m=nw(f)$.
\end{itemize}


Without loss of generality, we assume $r_1 \le r_2 $.
By \cite[Proposition 5.2, Lemma 5.10]{3ct} 
and  $\frac{1}{k} < \ct$, it follows that $r_1 < 2kn$ and $n\le 3k$.
Then, consider the weight $w_1=\frac{1}{n}(b^{*},dn-b^*,1,n)$ where $0<b^*<n$ and $bb^*\equiv 1$ (mod $n$). We will use frequently \cite[Lemma 2.1, Lemma 5.1]{3ct} to show the Claims below.
\begin{claim}\label{bound dn} If $a\ge 6k^2$, then $dn<4k$.
\end{claim} 
\begin{proof}[Proof of the Claim] We may assume $d>1$. Since $a\ge 6k^2$ and $b^*\ge 1$, we have $a\ge 2kn>r_1\ge \frac{r_1}{b^*}$.
Let $\frak{m}'\in f$ satisfy the weighted multiplicity $m_{1}=nw_{1}(\frak{m}')$. If the monomial $\frak{m}'$ doesn't involve the variable $y$, we see $m_1=nw_{1}(\frak{m}')\ge \frac{1}{a}nw(\frak{m}')\ge \frac{1}{a}m$. In particular, $\ct=\frac{a}{m}\le \frac{1}{m_1}\le \frac{a}{m}$ which is absurd. Thus, $\frak{m}'$ involves the variable $y$. Again, by $\frac{1}{k}<\ct=\frac{a}{m}\le \frac{1}{m_1}$, one sees $dn=dn-b^*+b^*\le m_1+b^*<k+n\le 4k$ and this verifies the claim.
\end{proof}

For our purpose,  we may assume $a>\max\{6k^2, \frac{r_1+dn^2-n}{dn-1}\}$.
Next, we compare the weight $w$ with the weight $w_{a-n}=\frac{1}{n}(r_1,r_2-dn^2,a-n,n)$.
Let $\frak{m}''=x^{l_1}y^{l_2}z^{l_3}u^{l_4}\in f$ satisfy the weighted multiplicity $m_{a-n}=nw_{a-n}(\frak{m}'')$.
\begin{claim} $\max\{l_2,l_3\}<k$ and either $l_2>0$ or $l_3>0$.
\end{claim}
\begin{proof}[Proof of the Claim]
From the assumption that $a>\frac{r_1+dn^2-n}{dn-1}$, we see $nw_{a-n}(y)=r_2-dn^2\ge a-n=nw_{a-n}(z)$.
Since $\frac{1}{k} < \ct\le \frac{a-n}{m_{a-n}}$, $\max\{l_2,l_3\}<k$.
Suppose that $l_2=l_3=0$. Then $m_{a-n}=nw_{a-n}(\frak{m}'')=nw(\frak{m}'')\ge nw(f)=m$ which leads to a contradiction $\frac{a-n}{m}\ge \frac{a-n}{m_{a-n}}\ge \ct=\frac{a}{m}$. Thus, either $l_2>0$ or $l_3>0$. This verifies the claim.
\end{proof}

\begin{claim}\label{squeezecAn} $dnl_2+l_3\ge k$ and $\frac{1}{dnl_2+l_3}\le \frac{a}{m}=\ct\le \frac{a-n}{(r_2-dn^2)l_2+(a-n)l_3}$.
\end{claim}
\begin{proof}[Proof of the Claim]
It is easy to see 
\[m_{a-n}=r_1l_1+(r_2-dn^2)l_2+(a-n)l_3+nl_4=nw(\frak{m}'')-(dn^2l_2+nl_3)\ge m-(dn^2l_2+nl_3).\]
Thus $$\frac{a}{m}=\ct\le \frac{a-n}{m_{a-n}}\le \frac{a-n}{m-(dn^2l_2+nl_3)}.$$ This gives  
$\frac{1}{dnl_2+l_3}\le \frac{a}{m}<\frac{1}{k-1}$. In particular, $dnl_2+l_3\ge k$.
On the other hand, we have $$m_{a-n}=nw_{a-n}(\frak{m}'')\ge nw_{a-n}(y^{l_2}z^{l_3})=(r_2-dn^2)l_2+(a-n)l_3.$$ 
Thus, 

\[ \ct\le \frac{a-n}{m_{a-n}}\le \frac{a-n}{(r_2-dn^2)l_2+(a-n)l_3}.\]
This finishes the proof of Claim \ref{squeezecAn}.
\end{proof} 
 
From Claim \ref{squeezecAn}, we note \[ \ct\le \frac{a-n}{(r_2-dn^2)l_2+(a-n)l_3}=\frac{1-\frac{n}{a}}{dnl_2+l_3-\frac{(r_1+dn^2)l_2+nl_3}{a}}\le \frac{1-\frac{n}{a}}{k-\frac{(r_1+dn^2)l_2+nl_3}{a}}.\]



Now, suppose $\{\ct_i=\frac{a_i}{m_i}\}$ is a sequence converging to a real number $x$ where each $\ct_i\in \mathcal{T}_{3,cA/n}^{\textup{can}}\cap (\frac{1}{k},\frac{1}{k-1})$ is realized as a weighted blow up with weights $\frac{1}{n}(r_{i1},r_{i2},a_i,1)$ with  $r_{i1} \le r_{i2}$ and $m_{a_i-n_i}=n_iw_{a_i-n_i}(\frak{m}_i'')$ for some monomial $\frak{m}_i''=x^{l_{i1}}y^{l_{i2}}z^{l_{i3}}u^{l_{i4}}\in f_i.$ 
By above discussions and passing to a subsequence, one may assume each $d_in_i=d'n', r_{i1}=r_1', l_{i2}=l_2', l_{i3}=l_3'$ for some fixed integers $d',n',r_1',l_2',l_3'$.
It follows from Claim \ref{squeezecAn} that $x=\lim_{a_i\to \infty} \frac{a_i}{m_i}=\frac{1}{d'n'l_2'+l_3'}=\frac{1}{k}.$ This completes the proof.
\end{proof}

The argument in \cite[Proposition 6.1]{3ct} allows us to have the generalization. 

\begin{prop}\label{cDacc}
Let $k\ge 2$ be an integer. Then $\mathcal{T}_{3,cD}^{\textup{can}}\cap (\frac{1}{k},\frac{1}{k-1})$
is a finite set. 
\end{prop}
\begin{proof}


Suppose that $\ct:=\ct(X,S)=\frac{a}{m} \in \mathcal{T}_{3,cD}^{\textup{can}}\cap (\frac{1}{k},\frac{1}{k-1})$ with $a\geq 5$ and computed by a divisorial contraction $\sigma$. \cite[Theorem 1.2]{Kawakita05} shows that $\sigma$ is classified by Case 1 and Case 2.

\noindent {\bf Case 1.}
Suppose $\sigma: Y \to X$  is a weighted blow up with weight $w=wt(x,y,z,u)=(r+1,r,a,1)$ with center $P\in X$ by the analytical identification:
\[ (P\in X)\simeq  o\in ( \varphi: x^2+xq(z,u)+y^2u+\lambda y z^2+\mu z^3+p(y,z,u)=0) \subset \bC^4,\]
where $o$ denotes the origin of $\bC^4$ and $2r+1=ad$ where $d\ge 3$ and $a$ is an odd integer.

\begin{claim}\label{cDdm0ACC1}
$d\leq 2k-1$ and $m< 4kr$.
\end{claim}

\begin{proof}[Proof of the Claim]
Let $s=\frac{d-1}{2}$ and $\sigma':Y'\to X$ be the weighted blow up of weights $w'=(s+1,s,1,1)$. 
By \cite[Lemma 2.1, Lemma 6.3]{3ct}, one sees the weighted multiplicity $m'=w'(f)\le k-1$ where the prime divisor $S$ is given by $f=0$.
Let $\frak{m}=x^{t_1}y^{t_2}z^{t_3}u^{t_4}\in f$ with $m'=w'(\frak{m})$. In particular, each $t_i<k$.
One sees $$m\leq w(\frak{m})=(r+1)t_1+rt_2+at_3+t_4<(2r+a+2)k<4kr.$$ 
If $d>2k-1$, we see $t_1=t_2=0$ and hence $$\frac{m}{a}\ge m'=w'(\frak{m})\ge \frac{1}{a}w(\frak{m})\ge \frac{m}{a},$$ a contradiction. 
Thus, $d\leq 2k-1$. This verifies the claim.
\end{proof}
We then consider the weighted blow up $\sigma_1\colon Y_1\to X$ (resp. $\sigma_2\colon Y_1\to X$) with weights $w_1=(d,d,2,1)$ (resp.  $w_2=(r+1-d,r-d,a-2,1)$). Note that the defining equation of the exceptional set of $\sigma_1$ is $x^2+\eta z^d$ for some nonzero constant $\eta$ as computation in \cite[Claim 6.6]{3ct} and hence the exceptional set of $\sigma_1$ is irreducible. By \cite[Lemma 6.3]{3ct}, 
the exceptional set of $\sigma_2$ is irreducible (see also \cite[Case Ic]{CJK15}).

Recall that $2r+1=ad$ and hence
\[w_1\succeq \frac{d}{r+1} w\ \ \textup{ and } \ \ \ w_2\succeq \frac{r-d}{r} w.\]
It follows from \cite[Lemma 2.1]{3ct} that \[\lfloor \frac{2}{a}m\rfloor \ge m_1\ge \lceil \frac{d}{r+1} m\rceil \ \ \textup{ and } \ \ \ \lfloor \frac{a-2}{a}m\rfloor \ge m_2\ge \lceil \frac{r-d}{r}m\rceil, \eqno{\dagger_1} \]
where $m_1:=w_1(f)$ and $m_2:=w_2(f)$ are the weighted multiplicities.
The conclusion is derived from the following.
\begin{claim}\label{r1bound} $r\le 8k^2$. 
\end{claim}
\begin{proof}[Proof of the Claim]
Suppose on the contrary that $r> 8k^2$. From Claim \ref{cDdm0ACC1}, one sees
$dm<4k(2k-1)r<8k^2r<(r+1)r$. Thus
\begin{align*}
& \lceil \frac{d}{r+1}m\rceil +\lceil \frac{r-d}{r}m\rceil \ge   \lceil \frac{d}{r+1}m +\frac{r-d}{r}m\rceil  = \lceil m-\frac{dm}{r(r+1)} \rceil=m.
\end{align*}

However, $a$ is odd and $a\nmid m$, hence  $\frac{2m}{a}$ is not an integer.
This implies that
\[\lfloor \frac{2}{a}m\rfloor+\lfloor \frac{a-2}{a}m\rfloor =m-1, \]
which contradicts to $\dagger_1$. This verifies Claim \ref{r1bound}.
\end{proof}

\noindent {\bf Case 2.} Suppose $\sigma$ is a weighted blow up with weight $w=(r+1,r,a,1,r+2)$ with center $P\in X$ by the analytical identification
$$o\in \left( \begin{array}{ll}
\varphi_{1} \colon x^2+yt+p(y,z,u)=0 ;\\
 \varphi_{2} \colon  yu+z^{d}+q(z,u)u+t=0 \\
\end{array} \right) \subset \mathbb{C}^5
$$
where $o$ denotes the origin of $\bC^5$ such that $r+1=ad$ where $d\ge 2$.

Compare the weight $w$ with the weights $w_1=(d,d,1,1,d)$ and $w_{a-1}=(r-d+1,r-d,a-1,1,r-d+2)$. 
By \cite[Lemma 2.1, Lemma 6.7, Lemma 6.8]{3ct}, we have
$$ \lfloor \frac{1}{a} m \rfloor  \ge \lceil \frac{d}{r+2} m \rceil \textup{ and } \lfloor \frac{a-1}{a} m \rfloor  \ge \lceil \frac{r-d}{r} m \rceil.   \eqno{\dagger_2}$$

\begin{claim}\label{cDcase2dm0ACC}
$d\le k-1$ and $m< 4kr$.
\end{claim}

\begin{proof}[Proof of the Claim]
By \cite[Lemma 2.1, Lemma 6.8]{3ct}, one sees the weighted multiplicity $m_1=w_1(f)\le k-1$ where the prime divisor $S$ is given by $f=0$.
Let $\frak{m}=x^{\alpha_1}y^{\alpha_2}z^{\alpha_3}u^{\alpha_4}t^{\alpha_5}\in f$ with $m_1=w_1(\frak{m})$. In particular, each $\alpha_i<k$.
One sees $$m\leq w(\frak{m})=(r+1)\alpha_1+r\alpha_2+a\alpha_3+\alpha_4+(r+2)\alpha_5<(3r+a+4)k<4kr.$$ 
If $d>k-1$, we see $\alpha_1=\alpha_2=\alpha_5=0$ and hence $$\frac{m}{a}\ge m_1=w_1(\frak{m})\ge \frac{1}{a}w(\frak{m})\ge \frac{m}{a},$$ a contradiction. 
Thus, $d\leq k-1$. This verifies the Claim.
\end{proof}

\begin{claim}\label{r2bound} $r\le 8k^2-2$. 
\end{claim}
\begin{proof}[Proof of the Claim]
Suppose on the contrary that $r> 8k^2-2$. From Claim \ref{cDcase2dm0ACC}, one sees
$2dm<8k^2r<(r+2)r$. Thus
\begin{align*}
& \lceil \frac{d}{r+2}m\rceil +\lceil \frac{r-d}{r}m\rceil \ge   \lceil \frac{d}{r+2}m +\frac{r-d}{r}m\rceil  = \lceil m-\frac{2dm}{r(r+2)} \rceil=m.
\end{align*}

However, $a\nmid m$, hence  
\[\lfloor \frac{m}{a}\rfloor+\lfloor \frac{a-1}{a}m\rfloor =m-1, \]
which contradicts to $\dagger_2$. The proof of Proposition \ref{cDacc} is completed.
\end{proof}
 
Proposition \ref{cDacc} is verified by Claim \ref{r1bound} (resp. \ref{r2bound}) in Case 1 (resp. 2). \end{proof}

Similarly, we have the generalization of \cite[Proposition 7.1]{3ct} as follows.

\begin{prop}\label{cD2acc}
Let $k\ge 2$ be an integer. Then the only possible accumulation point of $\mathcal{T}_{3,cD/2}^{\textup{can}}\cap (\frac{1}{k},\frac{1}{k-1})$ is $\frac{1}{k}$. 
\end{prop}
\begin{proof}
Suppose that $\ct:=\ct(X,S)=\frac{a}{m} \in \mathcal{T}_{3,cD/2}^{\textup{can}}\cap (\frac{1}{k},\frac{1}{k-1})$ with $a\geq 5$ and computed by a divisorial contraction $\sigma$. \cite[Theorem 1.2]{Kawakita05} shows that $\sigma$ is classified by Case 1 and Case 2.

\noindent {\bf Case 1.}
Suppose $\sigma$ is a weighted blow up $\sigma: Y \to X$ with weight $w=\frac{1}{2}(r+2,r,a,2)$ with center $P\in X$ by the analytical identification:
\[ o\in (\varphi: x^2+xzq(z^2,u)+y^2u+\lambda y z^{2\alpha -1}+p(z^2,u)=0 ) \subset \mathbb{C}^4/\frac{1}{2}(1,1,1,0)\]
where $o$ denotes the origin of $\mathbb{C}^4/\frac{1}{2}(1,1,1,0)$ such that $r+1=ad$ where both $a$ and $r$ are odd.

\begin{claim}\label{cDdm0ACC}
$d\leq k$ and $m< 4kr$.
\end{claim}

\begin{proof}[Proof of the Claim]
Let $s=d-1$ and $\sigma':Y'\to X$ be the weighted blow up of weights $w'=\frac{1}{2}(s+2,s,1,2)$. 
By \cite[Lemma 2.1, Lemma 7.3]{3ct}, one sees the weighted multiplicity $m'=2w'(f)\le k-1$ where the integral and effective divisor $S$ is given by $f=0$.
Let $\frak{m}=x^{t_1}y^{t_2}z^{t_3}u^{t_4}\in f$ with $m'=2w'(\frak{m})$. In particular, each $t_i<k$.
One sees $$m\leq 2w(\frak{m})=(r+2)t_1+rt_2+at_3+2t_4<(2r+a+2)k<4kr.$$ 
If $s=d-1>k-1$, we see $t_1=t_2=0$ and hence $$\frac{m}{a}\ge m'=w'(\frak{m})\ge \frac{1}{a}w(\frak{m})\ge \frac{m}{a},$$ a contradiction. 
Thus, $d\leq k$. This verifies the claim.
\end{proof}

Next, we compare the weight $w$ with the weight $w_{a-2}=\frac{1}{2}(r-2d+2,r-2d,a-2,2)$.
Let $\frak{m}''=x^{l_1}y^{l_2}z^{l_3}u^{l_4}\in f$ satisfy the weighted multiplicity $m_{a-2}=2w_{a-2}(\frak{m}'')$.
\begin{claim}\label{123} $\max\{l_1,l_2,l_3\}<k$ and at least one of $l_1,l_2, l_3$ is nonzero.
\end{claim}
\begin{proof}[Proof of the Claim]
Since $\frac{1}{k} < \ct\le \frac{a-2}{m_{a-2}}$, one sees \begin{align*} (a-2)k&>m_{a-2}=2w_{a-2}(\frak{m}'')\\&\ge 2w_{a-2}(x^{l_1}y^{l_2}z^{l_3})=(r-2d+2)l_1+(r-2d)l_2+(a-2)l_3.
\end{align*} As $r-2d=ad-2d-1\ge a-2-1$ and $a\ge 5$, one has $l_1+l_2+l_3<k$.
Suppose that $l_2=l_3=0$. Then $m_{a-n}=nw_{a-n}(\frak{m}'')=nw(\frak{m}'')\ge nw(f)=m$ which leads to a contradiction $\frac{a-2}{m}\ge \frac{a-2}{m_{a-2}}\ge \ct=\frac{a}{m}$. Thus, at least one of $l_1, l_2, l_3$ is positive. This verifies the claim.
\end{proof}

\begin{claim}\label{squeezecD2} $dl_1+dl_2+l_3\ge k$ and $\frac{1}{dl_1+dl_2+l_3}\le \frac{a}{m}\le \frac{a-2}{(r-2d+2)l_1+(r-2d)l_2+(a-2)l_3}$.
\end{claim}
\begin{proof}[Proof of the Claim]
It is easy to see 
\begin{align*}
m_{a-2}&=(r-2d+2)l_1+(r-2d)l_2+(a-2)l_3+2l_4\\ &=2w(\frak{m}'')-2(dl_1+dl_2+l_3) \ge m-2(dl_1+dl_2+l_3).
\end{align*}
Thus $$\frac{a}{m}=\ct\le \frac{a-2}{m_{a-2}}\le \frac{a-2}{m-2(dl_1+dl_2+l_3)}.$$ This gives 
$\frac{1}{dl_1+dl_2+l_3}\le \frac{a}{m}<\frac{1}{k-1}$. In particular, $dl_1+dl_2+l_3\ge k$.
On the other hand, we have $$m_{a-2}=2w_{a-2}(\frak{m}'')\ge 2w_{a-2}(x^{l_1}y^{l_2}z^{l_3})=(r-2d+2)l_1+(r-2d)l_2+(a-2)l_3.$$ 
Thus, 

\[ \ct\le \frac{a-2}{m_{a-2}}\le \frac{a-2}{(r-2d+2)l_1+(r-2d)l_2+(a-2)l_3}.\]
This finishes the proof of Claim \ref{squeezecD2}.
\end{proof} 
 
From Claim \ref{squeezecAn}, we note \begin{align*} \ct&\le \frac{a-2}{(r-2d+2)l_1+(r-2d)l_2+(a-2)l_3} =\frac{1-\frac{2}{a}}{dl_1+dl_2+l_3-\frac{(2d-1)l_1+3l_2+2l_3}{a}} \\ &\le \frac{1-\frac{2}{a}}{k-\frac{(2d-1)l_1+3l_2+2l_3}{a}}.\end{align*}

Now, suppose $\{\ct_i=\frac{a_i}{m_i}\}$ is a sequence converging to a real number $x$ where each $\ct_i:=\ct(X_i,S_i)\in \mathcal{T}_{3,cD/2}^{\textup{can}}\cap (\frac{1}{k},\frac{1}{k-1})$ is realized as a weighted blow up with weights $\frac{1}{2}(r_{i}+2,r_{i}, a_i, 2)$ with $r_i+1=a_id_i$ and $m_{a_i-2}=2w_{a_i-2}(\frak{m}_i'')$ for some monomial $\frak{m}_i''=x^{l_{i1}}y^{l_{i2}}z^{l_{i3}}u^{l_{i4}}\in f_i$ where $f_i$ denotes a formal power series defining $S_i$.
By above discussions and passing to a subsequence, one may assume each $d_i=d', l_{i1}=l_1', l_{i2}=l_2', l_{i3}=l_3'$ for some fixed integers $d',l_1',l_2',l_3'$.
It follows from Claim \ref{squeezecD2} that $x=\lim_{a_i\to \infty} \frac{a_i}{m_i}=\frac{1}{d'l_1'+d'l_2+l_3'}=\frac{1}{k}.$

\noindent {\bf Case 2.}
Suppose $\sigma$ is a weighted blow up with weight $w=\frac{1}{2}(r+2,r,a,2,r+4)$ with center $P\in X$ by the analytical identification:
$$o\in \left( \begin{array}{ll}
\varphi_{1}:= x^2+yt+p(z^2,u)=0 \\
 \varphi_{2}:= yu+z^{2d+1}+q(z^{2},u)zu+t=0 \\
\end{array} \right) \textup{ in } \mathbb{C}_{x,y,z,u,t}^5/\frac{1}{2}(1,1,1,0,1),$$ where $o$ denotes the origin of $\mathbb{C}_{x,y,z,u,t}^5/\frac{1}{2}(1,1,1,0,1)$ such that $r+2=a(2d+1)$ where $d$ is a positive integer.

Compare the weight $w$ with the weights $w_1=\frac{1}{2}(2d+1,2d+1,1,2,2d+1)$ and $w_{a-1}=\frac{1}{2}(r-2d+1,r-2d-1,a-1,2,r-2d+3)$.
By \cite[Lemma 2.1, Lemma 7.6, Lemma 7.7]{3ct}\footnote{The relation $r'+2=a'd$ in the statement of \cite[Lemma 7.7]{3ct} should be replaced by $r'+2=a'(2d+1)$. It is a typo.}, we have
$$ \lfloor \frac{1}{a} m \rfloor  \ge \lceil \frac{2d+1}{r+4} m \rceil \textup{ and } \lfloor \frac{a-1}{a} m \rfloor  \ge \lceil \frac{r-2d-1}{r} m \rceil. \eqno{\dagger_3}$$

\begin{claim}\label{cD2case2dm0ACC}
$2d+1\le k-1$ and $m< 4kr$.
\end{claim}

\begin{proof}[Proof of the Claim]
By \cite[Lemma 2.1, Lemma 7.6]{3ct}, one sees the weighted multiplicity $m_1=w_1(f)\le k-1$ where the integral and effective divisor $S$ is given by $f=0$.
Let $\frak{m}=x^{\alpha_1}y^{\alpha_2}z^{\alpha_3}u^{\alpha_4}t^{\alpha_5}\in f$ with $m_1=2w_1(\frak{m})$. In particular, each $\alpha_i<k$.
One sees $$m\leq 2w(\frak{m})=(r+2)\alpha_1+r\alpha_2+a\alpha_3+2\alpha_4+(r+4)\alpha_5<(3r+a+6)k<4kr.$$ 
If $2d+1>k-1$, we see $\alpha_1=\alpha_2=\alpha_5=0$ and hence $$\frac{m}{a}\ge m_1=2w_1(\frak{m})\ge \frac{1}{a}2w(\frak{m})\ge \frac{m}{a},$$ a contradiction. 
Thus, $2d+1\leq k-1$. This verifies the claim.
\end{proof}

\begin{claim} $r\le 16k^2-4$. 
\end{claim}
\begin{proof}[Proof of the Claim]
Indeed, suppose $r> 16k^2-4$. From Claim \ref{cD2case2dm0ACC}, one sees
$(8d+4)m<16k^2r<r(r+4)$. Thus
\begin{align*}
& \lceil \frac{2d+1}{r+4} m\rceil +\lceil \frac{r-2d-1}{r}m\rceil \ge   \lceil \frac{2d+1}{r+4} m+ \frac{r-2d-1}{r}m\rceil = \lceil m-\frac{(8d+4)m}{r(r+4)} \rceil=m.
\end{align*}

However, $a\nmid m$, hence  
\[\lfloor \frac{m}{a}\rfloor+\lfloor \frac{a-1}{a}m\rfloor =m-1, \]
which contradicts to $\dagger_3$.
 \end{proof}
This completes the proof of Proposition \ref{cD2acc}.

\end{proof}

\begin{rem} In Case 1 of the argument in Proposition \ref{cD2acc}, consider the weighted blow up $\sigma_1\colon Y_1\to X$ (resp. $\sigma_2\colon Y_2\to X$) with weights $w_1=\frac{1}{2}(2d,2d,2,2)$ (resp.  $w_2=\frac{1}{2}(r+2-2d,r-2d,a-2,2)$). By \cite[Lemma 7.3]{3ct}, $\sigma_2$ has irreducible exceptional divisor. If the inequality $\frac{2}{m_1}\geq \ct(X,S)$ holds where $m_1=2w_1(f)$ denotes the weighted multiplicity of $f$ which defines $S$, then we obtain upper bounds of $d$ and $r$ in terms of $k$ as in Case 2. In this case, the set $\mathcal{T}_{3,cD/2}^{\textup{can}}\cap (\frac{1}{k},\frac{1}{k-1})$ is finite. 
\end{rem}



\section{Proof of Theorem \ref{ACCctpair}}

We adopt notions in \cite[Sections 1,2]{3ct}. In this section, by abuse of notation, we denote by each $S_i$ an $\bQ$-Cartier effective divisor on $\bQ$-factorial projective 3-fold $X_i$. 

Suppose on the contrary that there exists an infinite strictly increasing sequence $\{\ct_i(X_i,B_i;S_i)\}$.
Note that it is known that each $\ct_i(X_i,B_i;S_i)$ is realized as a divisorial contraction $\sigma_i\colon Y_i\to X_i$ (See, for example, \cite{Corti} or \cite[Proposition 13-1-8]{Matsuki}). 
Let $B_i'$ be the strict transform of $B_i$ in $Y_i$. Let $R_i$ be an extremal ray with $(K_{Y_i}+{B_i}')\cdot R_i<0$. 
Since $B_i'$ is effective and $\text{Exc}(\sigma_i)$ is the exceptional divisor, $B_i'\cdot R_i\ge 0$ and thus $K_{Y_i}\cdot R_i<0$. So $\sigma_i$ is a $K_{Y_i}$-negative extremal divisorial contraction.
By passing to a subsequence, we may assume every center $Z(\sigma_i)$ of $\sigma_i$ is a point and of the same type. 

Denote by $N(B_i)$ (resp. $N(S_i)$) the number of irreducible components of $B_i$ (resp. $S_i$). 
That is, $B_i=\sum_{k=1}^{N(B_i)}b_{ik}B_{ik}$ (resp. $S_i=\sum_{k=1}^{N(S_i)}s_{ik}S_{ik}$) where each $B_{ik}$ (resp. $S_{ik}$) is an integral divisor and coefficient $b_{ik}$ (resp. $s_{ik}$) is positive.
Write 
$$K_{Y_i}=\sigma_i^*K_{X_i}+a_iE_i, B_{ik}'=\sigma_i^*B_{ik}-p_{ik}E_i, \textup{ and }S_{ik}'=\sigma_i^*S_{ik}-m_{ik}E_i.$$
We have $\ct_i(X_i,B_i;S_i)=\frac{a_i-\sum_{k=1}^{N(B_i)} b_{ik}p_{ik}}{\sum_{k=1}^{N(S_i)} s_{ik}m_{ik}}$. 
By multiplying with the index of the center, we assume that $a_i$ is the weighted discrepancy (resp. $p_{ik},m_{ik}$ are weighted multiplicities).

For our purpose, we assume each $B_i$ (resp. $S_i$) contains the center $Z(\sigma_i)$.  
Let $q$ be a positive integer with $1/q\le \ct_1(X_1,B_1;S_1)$ and let $I_b>0$ (resp. $J_b>0$) be the minimal element of the DCC set $I-\{0\}$ (resp. $J$). 
\begin{claim} \label{boundn} For all $i$, $N(B_i)\le \frac{2}{I_b}$ and $N(S_i)\le \frac{2q}{J_b}$. 
\end{claim}
Indeed, let $\sigma_i^1$ denote a divisorial contraction with minimal discrepancy with center $Z(\sigma_i)$. Let $m_{ik}^1$ (resp. $p_{ik}^1$) denote the corresponding weighted multiplicity. Suppose that $Z(\sigma_i)$ is a smooth point. Then $\sigma_i^1$ is the usual blow up and  
$$1/q\le \ct_i(X_i,B_i;S_i)\le \frac{2-\sum_{k=1}^{N(B_i)} b_{ik}p_{ik}^1}{\sum_{k=1}^{N(S_i)} s_{ik}m_{ik}^1}\le \frac{2}{\sum_{k=1}^{N(S_i)} s_{ik}m_{ik}^1}\le \frac{2}{\sum_{k=1}^{N(S_i)}J_b\cdot 1}$$ and $$0<2-\sum_{k=1}^{N(B_i)} b_{ik}p_{ik}^1\le 2-\sum_{k=1}^{N(B_i)} I_b\cdot 1.$$ If $Z(\sigma_i)$ is of other type, the weighted discrepancy of $\sigma_i^1$ is $1$ and the same argument holds. This finishes the proof of the Claim.

By Claim \ref{boundn} and passing to subsequences, we may assume for all $i$, $N(B_i)=k_1$ and $N(S_i)=k_2$ for some fixed integers $k_1$ and $k_2$. 
As both $I$ and $J$ are the DCC sets, we may further assume that for every $k$, the sequence $\{b_{ik}\}$ (resp. $\{s_{ik}\}$) is non-decreasing (cf. \cite[Lemma 2.3]{AM}).
Since $S_i$ and $B_i$ are effective, by passing to a subsequence, \cite[Theorem 1.5]{3ct} and Lemma \ref{boundpairindex} below, we may assume all center of $\sigma_i$ share the same index. 
Moreover, by passing to subsequences, we may assume 
the sequence of Newton polytopes $\{\Gamma^+(g_{ik})\}$ with $B_{ik}=(g_{ik}=0)$(resp. $\{\Gamma^+(f_{ik})\}$ with $S_{ik}=(f_{ik}=0)$) is non-increasing for all $k=1,...,k_1$ (resp. $k=1,...,k_2$).

Recall that it follows from \cite[Theorem 1.5]{3ct} that each $\sigma_i\colon Y_i\to X_i$ is a weighted blow up with weight $w_i$.
Now, for all integers $i<j$, we are able to choose the weight $w^i_j$ satisfying the following.
\begin{enumerate}
 \item 
 the  weighted multiplicities satisfies $n_iw_i(g_{ik}) \le n_iw_i(g_{jk})$ for all $k=1,...,k_1$ (resp. $n_iw_i(f_{ik}) \le n_iw_i(f_{jk})$ for all $k=1,...,k_2$). 
 \item  $n_iw_i \preceq n_jw^i_j$ where $n_i=n_j$ is the index of the center $Z(\sigma_i)$. ;
    \item  the weighted blow up $\sigma^i_j:Y^i_j\to X_j$ with weight $w^i_j$ over the point $P_j\in X_j$ has irreducible exceptional divisor, denoted by $E^i_j$. Then in this situation,  $K_{Y^i_j}=\sigma^{i*}_j K_{X_j}+\frac{a_i}{n_j}E^i_j$. (cf. \cite[Lemmas 4.1, 5.1, 6.3, 6.8, 7.3, 7.7]{3ct})
  \end{enumerate}

Combining (1) with (2), one sees for all $k$, $$m_{ik}=n_i w_i(f_{ik}) \leq n_i w_i(f_{jk})
\le n_j w^i_j(f_{jk})\ ;\ p_{ik}=n_i w_i(g_{ik}) \leq n_i w_i(g_{jk})
\le n_j w^i_j(g_{jk}).$$
Moreover by (3), $E^i_j$ defines a valuation on $X_j$ and computation on $E^i_j$ gives $\frac{a_i-\sum_{k=1}^{k_1} b_{ik} n_j w^i_j(g_{jk})}{\sum_{k=1}^{k_2} s_{ik}n_jw^i_j(f_{jk})}  \ge \ct(X_j, B_j; S_j)$.
Above inequalities yields \[\ct(X_i, B_i; S_i)=\frac{a_i-\sum_{k=1}^{k_1} b_{ik}p_{ik}}{\sum_{k=1}^{k_2} s_{ik}m_{ik}}
\ge \frac{a_i-\sum_{k=1}^{k_1} b_{ik} n_j w^i_j(g_{jk})}{\sum_{k=1}^{k_2} s_{ik}n_jw^i_j(f_{jk})}\ge \ct(X_j, B_j; S_j),\]
which is the desired contradiction and we complete the proof of Theorem \ref{ACCctpair}.

\begin{lem}\label{boundpairindex} Keep notations above.
Write $B=\sum_{k=1}^{k_1} b_kB_k$ and $S=\sum_{k=1}^{k_2} s_k S_k$. Suppose that $\ct(X,B;S)=\frac{a-\sum_{k=1}^{k_1}b_kp_k}{\sum_{k=1}^{k_2}s_km_k} >   \frac{1}{q}$ where $q\ge 2$ is an integer and the center computing canonical threshold is of type $cA/n$ and each integral divisor $B_k$ (resp. $S_k$) contains the center. Then either $n \leq 3\max\{\frac{1}{b_1},...,\frac{1}{b_{k_1}},\frac{q}{s_1},...,\frac{q}{s_{k_2}}\}$ or $\ct(X,B;S)=\frac{1-\sum_{k=1}^{k_1}b_kt_k}{\sum_{k=1}^{k_2}s_k l_k}$ for some positive integers $t_1,...,t_{k_1},l_1,...,l_{k_2}$.
\end{lem}
\begin{proof}
This proof is a generalization of \cite[Lemma 5.10]{3ct}.

Suppose that  $n>3\max\{\frac{1}{b_1},...,\frac{1}{b_{k_1}},\frac{q}{s_1},...,\frac{q}{s_{k_2}}\}$.
Consider the weights $w=\frac{1}{n}(r_1,r_2,a,n)$ and $w_3=\frac{1}{n}(r_1',r_2',3,n)$ satisfying $r_1+r_2=adn$, $r_1'+r_2'=3dn$, $\min\{r_1',r_2'\}>n$ and $a\equiv br_1$ (mod $n$) and $3\equiv br_1'$ (mod $n$). 
Note that exceptional set of the weighted blow up of weight $w_3$ is an irreducible divisor (cf. \cite[Lemma 5.1]{3ct}). 
Denote by $p_{3k}:=nw_3(g_k)$ and $m_{3k}:=nw_3(f_k)$ the weighted multiplicities where $B_k$ (resp. $S_k$) is defined by $g_k=0$ (resp. $f_k=0$).
Note that $$ \frac{3-\sum_{k=1}^{k_1}b_kp_{3k}}{\sum_{k=1}^{k_2}s_km_{3k}} \ge\ct(X,B;S)>\frac{1}{q}.$$ 
One sees for every $k=1,...,k_2$ (resp. $k=1,...,k_1$), $$\min\{r_1',r_2'\}>n>\frac{3q}{s_{k}}>m_{3{k}}\textup{ (resp.  } 3-b_{k}p_{3{k}}\ge 3-\sum_{k=1}^{k_1}b_kp_{3k}>0\textup{)}.$$ 
This implies that there exists $z^{l_{k}}\in f_{k}$ (resp. $z^{t_{k}}\in g_k$) such that  $m_{3k}=n w_3(z^{l_k})$ (resp. $p_{3k}=nw_3(z^{t_k})$). This implies in particular that $$ m_k=nw(f_k)\le  nw(z^{l_k})=al_k , \textup{ (resp.  }p_k=nw(g_k)\le  nw(z^{t_k})=at_k \textup{ )}.$$
Thus, $$\frac{1-\sum_{k=1}^{k_1}b_kt_k}{\sum_{k=1}^{k_2}s_k l_k}=\frac{3-\sum_{k=1}^{k_1}b_kp_{3k}}{\sum_{k=1}^{k_2}s_km_{3k}} \ge\ct(X,B;S)=\frac{a-\sum_{k=1}^{k_1}b_kp_k}{\sum_{k=1}^{k_2}s_km_k} \ge\frac{1-\sum_{k=1}^{k_1}b_kt_k}{\sum_{k=1}^{k_2}s_k l_k}.$$
\end{proof}

\end{document}